\documentclass[11pt,a4paper]{article}

\usepackage{psfrag}
\usepackage{amsmath,amssymb,cite}
\usepackage{latexsym}
\usepackage{graphicx}
\usepackage{lscape}
\usepackage{afterpage}
\usepackage{hyperref}

\usepackage{xcolor,soul}

\hypersetup{
    colorlinks=true,
    linkcolor=blue,
    filecolor=magenta,
    urlcolor=cyan,
    citecolor=red,
}

\newcommand{\Fb}{\overline{F}}
\newcommand{\xb}{\overline{x}}

\newcommand{\ds}{\displaystyle}
\newcommand{\nexto}{\kern -0.54em}
\newcommand{\dR}{{\rm {I\ \nexto R}}}

\newcommand{\dZ}{{\cal Z \kern -0.7em Z}}
\newcommand{\dC}{{\rm\hbox{C \kern-0.8em\raise0.2ex\hbox{\vrule
height5.4pt width0.7pt}}}}
\newcommand{\dQ}{{\rm\hbox{Q \kern-0.85em\raise0.25ex\hbox{\vrule
height5.4pt width0.7pt}}}}
\newcommand{\proofbox}{\hspace{\fill}{$\Box$}}

\newtheorem{theorem}{Theorem}

\newtheorem{remark}{Remark}

\newtheorem{example}{Example}

\newenvironment{proof}{Proof.}{\proofbox}

\begin{document}

\author{Authors}

\author{
C. Yal{\c c}{\i}n Kaya\footnote{Mathematics, UniSA STEM, University of South Australia, Mawson Lakes, S.A. 5095, Australia. E-mail: yalcin.kaya@unisa.edu.au\,.}
\and
J. Lyle Noakes\footnote{Department of Mathematics and Statistics, University of Western Australia, Crawley, W.A. 6009, Australia. E-mail: lyle.noakes@uwa.edu.au\,.}
\and
Erchuan Zhang\footnote{Department of Mathematics and Statistics, University of Western Australia, Crawley, W.A. 6009, Australia. E-mail: erchuan.zhang@uwa.edu.au\,.}
}

\title{\vspace{-10mm}\bf Multi-objective Variational Curves}

\maketitle

\begin{abstract} {\noindent\sf }
Riemannian cubics in tension are critical points of the linear combination of two objective functionals, namely the squared norms of the velocity and acceleration of a curve on a Riemannian manifold.  We view this variational problem of finding a curve as a multi-objective optimization problem and construct the Pareto fronts for some given instances where the manifold is  a sphere and where the manifold is a torus.  The Pareto front for the curves on the torus turns out to be particularly interesting: the front is disconnected and it reveals two distinct Riemannian cubics with the same boundary data, which is the first known nontrivial instance of this kind.  We also discuss some convexity conditions involving the Pareto fronts for curves on general Riemannian manifolds.

%\textcolor{red}{{\em Riemannian cubics in tension} are critical points of minimizing the linear combination of squared norm of velocities and accelerations of curves. From the view point of multi-objective optimization, the optimization problem is the {\em weighted-sum scalarization} of squared norm of velocities and that of accelerations, which may not be able to (or partially) recover the {\em Pareto front} of the two objective functionals: squared norm of velocities and accelerations. Motivated by this, we, in this manuscript, consider the multi-objective variational curves, i.e., minimising these two quantities of curves at the same time using other scalarization techniques such as {\em Chebyshev scalarization}. We discuss the convexity conditions of the epigraph of the Pareto front and present some simulations of the multi-objective variational curves on the sphere and the torus endowed with the standard Euclidean metric.}
\end{abstract}

\begin{verse}
\end{verse}
\begin{verse}
\end{verse}

%\centerline{\bf Submitted to }
\pagestyle{myheadings}
%\markboth{}{Submitted to {\sf }, }
\markboth{}{\sf\scriptsize Multi-objective Variational Curves\ \ by C. Y. Kaya, J. L. Noakes \& E. Zhang}

\section{Introduction}

Variational curves are curves which minimize an objective functional such as the length or some other quantity obtained from its velocity or acceleration, and so on.  Very often more than just one of these (competing) functionals are to be minimized simultaneously, giving rise to a multi-objective optimization problem.  A typical example is cubic curves in tension, or simply {\em cubics in tension},  where a critical point of a linear combination of the squared $L^2$-norms of the velocity and the acceleration is found \cite{NoaPop2006,Noakes2005,Zhang2008}.  Although, clearly, two objective functionals are involved here, this problem is not treated as a multi-objective optimization problem in the literature in that it is only posed as a (single-objective) {\em weighted-sum scalarization} of the two objective functionals.

The aim of a multi-objective optimization problem is to find the set of all trade-off solutions.  A trade-off solution, or a {\em Pareto solution}, is a solution which cannot be improved any further to make an objective functional value better  (smaller), without making some of the other objective functional values worse (larger).  Such a solution set is referred to as the {\em Pareto set}, or the {\em efficient set}, and in the objective value space, the {\em Pareto front}, or the {\em efficient front}, respectively; see for example~\cite{Eichfelder2008,Miettinen1999}.

It is of great interest to construct the Pareto front for inspection, so that a desirable solution can be selected by the practitioner.  The most popular approach in constructing a Pareto front is first to scalarize the problem via parameters, or {\em weights}, so as to obtain a continuum of single-objective optimization problems.  These parametric problems are then solved to get a solution in the Pareto front for each parameter value~\cite{DutKay2011,Eichfelder2008,Miettinen1999}.  

To be able to construct the whole Pareto front the parametric scalarization is required to be a surjection from the space of weights to the set of Pareto solutions.  If the problem is convex, i.e., the objective functionals and the constraint set are convex, then the weighted-sum scalarization suffices to furnish a surjection~\cite{BonKay2010}.  If, on the other hand, the problem is nonconvex, the weighted-sum scalarization is not guaranteed to be surjective anymore.

It is well-known that variational problems can often be cast as optimal control problems; see for example~\cite{KayNoa2013}.  In~\cite[Theorem~1]{KayMau2014}, it is stated that the so-called {\em Chebyshev scalarization} (the weighted max-norm of the objective functionals) is surjective, guaranteeing construction of the whole Pareto front. 

In this paper, we consider variational curves with two objective functionals to minimize, motivated by cubics in tension.  We state a similar result to that in~\cite[Theorem~1]{KayMau2014}. We construct the Pareto fronts of certain instances of two Riemannian manifolds, namely a sphere and a torus.  To the best knowledge of the authors, this is the first time multi-objective variational curves on Riemannian manifolds are studied.

We observe that, on the torus, the Pareto front generated %
%with respect to 
by using
the Chebychev scalarization is disconnected and Riemannian cubics with the same boundary conditions are not unique.  Non-uniqueness of Riemannian cubics under given nontrivial boundary conditions as presented in this paper has not been encountered in the literature before.

The paper is organized as follows. In Section~\ref{problem}, we introduce the multi-objective optimization problem
%, i.e., minimising the vector of total squared norm of velocity and that of acceleration of a curve on a Riemannian manifold, which is 
motivated by 
%the problem of Riemannian 
cubics in tension. Then we sketch the scalarization techniques, namely the Chebychev and weighted-sum scalarizations.
%, that turn a multi-objective to a single-objective optimization problem. 
In Section \ref{sec:convex}, we present some necessary and sufficient conditions for the epigraph of Pareto fronts to be convex on general Riemannian manifolds. Numerical experiments 
%are shown 
in Section \ref{sec:num} 
illustrate the Pareto fronts and Pareto curves
on a sphere and torus endowed with the standard Euclidean metric. 
%We observe that, on the torus, the Pareto front generated with respect to the Chebychev scalarization is disconnected and Riemannian cubics with the same boundary conditions are not unique. 
Finally, we conclude the paper in Section \ref{sec:con} and present some future directions.

%%%%%%%%%%%%%%%%%%%%%%%%%%%%%%%%%%%%%%%%%%%%%
\section{Problem Statement and Preliminaries}
\label{problem}

We are motivated by the problem of finding {\em Riemannian cubics in tension}. Let $M$ be a path-connected Riemannian manifold with a Riemannian metric $\langle\cdot,\cdot\rangle$ and a Levi--Civita connection $\nabla$, $TM$ the tangent bundle of $M$. Given two points $x_0,x_1$ and their associated velocities $v_0\in T_{x_0}M$, $v_1\in T_{x_1}M$, we denote by $\mathcal{X}$ the space of all smooth curves $x:[t_0,t_f]\rightarrow M$ satisfying $x(t_0)=x_0$, $\dot{x}(t_0)=v_0$, $x(t_f)=x_1$, $\dot{x}(t_f)=v_1$, where $t_0$ and $t_f$ are fixed. Riemannian cubics in tension are defined to be  critical points of the functional
\begin{equation}  \label{CIT}
\int_{t_0}^{t_f} \bigg(\tau\,\langle\dot{x}(t),\dot{x}(t)\rangle + \langle\nabla_t\dot{x}(t),\nabla_t\dot{x}(t)\rangle\bigg)\,dt\,,
\end{equation}
where $x\in \mathcal{X}$, $\tau>0$ is a constant. Equivalently, Riemannian cubics in tension are solutions of the following Euler--Lagrange equations \cite{NoaPop2006,Noakes2005,Zhang2008}
\begin{align}\label{cubELeq}
\nabla_t^3\dot{x}(t)+R(\nabla_t\dot{x}(t),\dot{x}(t))\dot{x}(t)-\tau\nabla_t\dot{x}(t)={\bf 0}
\end{align}
over $\mathcal{X}$, where $R$ is the Riemannian curvature tensor associated with the Levi-Civita connection $\nabla$, i.e., $R(X,Y)Z=\nabla_X\nabla_YZ-\nabla_Y\nabla_XZ-\nabla_{[X,Y]}Z$ for vector fields $X,Y,Z$ on $M$.

We pose the problem of finding a critical point of the above functional in a more general form as a {\em multi-objective optimization} problem:
\[
\mbox{(MOP)}\qquad\min_{x\in \mathcal{X}}\ \left[\int_{t_0}^{t_f} \langle\dot{x}(t),\dot{x}(t)\rangle\,dt\,,\ \ \int_{t_0}^{t_f} \langle\nabla_t\dot{x}(t),\nabla_t\dot{x}(t)\rangle\,dt\right]\,.
\]
%This is a more general formulation because its solution set is in general bigger.
Let
\[
F_1(x) := \int_{t_0}^{t_f} \langle\dot{x}(t),\dot{x}(t)\rangle\,dt \qquad\mbox{and}\qquad
F_2(x) := \int_{t_0}^{t_f} \langle\nabla_t\dot{x}(t),\nabla_t\dot{x}(t)\rangle\,dt\,.
\]
Define the vector of objective functionals, $F(x) := (F_1(x),F_2(x))$.  The point $x^*\in \cal{X}$ is said to be a {\em Pareto minimum} if there
exists no $x\in \cal{X}$ such that $F(x) \not= F(x^*)$ and
\[
F_i(x) \le F_i(x^*)\quad\mbox{for } i=1,2\,.
\]
On the other hand, $x^*\in \cal{X}$ is said to be a {\em weak Pareto minimum} if there exists no $x\in \cal{X}$ such that
\[
F_i(x) < F_i(x^*)\quad\mbox{for } i=1,2\,.
\]
The set of all vectors of objective functional values at the weak Pareto minima is said to be the {\em Pareto front} (or {\em efficient front}) of Problem~(MOP) in the 2-dimensional {\em objective value space}.  Note that the coordinates of a point in the Pareto front are simply $F_i(x^*)$, $i=1,2$.  The Pareto front is usually a curve (connected or disconnected).

\subsection{Scalarization}
\label{scalarization}

For computing a solution of the nonconvex multi-objective problem (MOP), we will consider the following single-objective problem, referred to as {\em scalarization}.  While this scalarization is often considered for nonconvex multi-objective finite-dimensional optimization problems~\cite{Miettinen1999}, it has also been considered for infinite-dimensional (optimal control) problems in \cite{KayMau2014,KayMau2023}.
\[
\mbox{(MOP$_w$)}\qquad \min_{x\in \cal{X}}\ \max\{w_1\,F_1(x),\, w_2\,F_2(x)\}\,,
\]
where $w_1$ and $w_2$ are referred to as {\em weights}.  Problem~(MOP$_w$) is referred to as the {\em weighted Chebychev problem} (or {\em Chebychev scalarization})\footnote{The more commonly used weighted-sum scalarization will be considered below.}. It is required that $w_1,w_2\geq 0$.  For brevity in notation, we set $w_1 = w$ and $w_2 = 1 - w$ with $w\in[0,1]$. 

The following theorem is a direct consequence of \cite[Corollary
5.35]{Jahn2011}. It can also be concluded by generalizing the proofs
given in a finite-dimensional setting in \cite[Theorems 3.4.2 and
3.4.5]{Miettinen1999} to the infinite-dimensional setting in a
straightforward manner.

\begin{theorem} \label{scalar_theorem}
The point $x^*$ is a weak Pareto minimum of {\rm (MOP)} if, and only if, $x^*$ is a solution of {\rm (MOP$_w$)} for some $w_1,w_2>0$, $w_1 + w_2 = 1$.
\end{theorem}

An {\em ideal cost} $F^*_i$, $i=1,2$, associated with Problem~(MOP$_w$) is the optimal value of the optimal control problem,
\[
\mbox{(P$_i$)}\qquad \inf_{x\in \cal{X}}\ F_i(x)\,.
\]

\begin{remark} \rm
%[Consider using $\inf$ instead of $\min$ everywhere?]\ \ \ \Erchuan{Note that the optimization problem $\min_{x\in \mathcal{X}}F_1(x)$, which is sufficient to specify two endpoints in general, may not be well-posed. However, this issue can be avoided when considering the scalarization of the problem (MOP) with the weight $w\neq 1$. Alternatively, we can add a small regularization term $\epsilon F_2(x)$, where $\epsilon>0$ is a small constant, to the first cost function in order to make it well-defined within the set $\mathcal{X}$. To simplify our discussions and notations, we assume the problem (MOP) is well defined under the constraint $x\in\mathcal{X}$.}
Since $F_i(x)$ for $x\in\mathcal{X}$ is bounded below, the problem $(P_i)$ is well-defined. Note that for the optimization problem $\min_{x\in \mathcal{X}}F_1(x)$ it is customary to specify two endpoints, but problems ($P_1$) and ($P_2$) share the same set of inputs. For the optimization problem  $\min_{x\in \mathcal{X}}F_2(x)$, we usually specify two endpoints and their associated end-velocities which usually determines a Riemannian cubic curve. 
The difficulty is that $F_1$ usually does not have a minimizer when the curves are restricted in this way, which is why we ask for an infimum instead of a minimizer. 
\proofbox
\end{remark}

Let $\xb$ be a minimizer of the single-objective problem (P$_i$) above.  Then we also define $\Fb_j := F_j(\xb)$, for $j\neq i$ and $j=1,2$.

In general, it is useful to add a positive number to each
objective in order to obtain an even spread of the Pareto points
approximating the Pareto front---see, for example, \cite{DutKay2011},
for further discussion and geometric illustration.  For this purpose,
we define next the so-called {\em utopian objective values}.

A {\em utopian objective vector} $\beta^*$ associated with
Problem~(MOP$_w$) consists of the components $\beta^*_i$ given as
$\beta^*_i = F_i^* - \varepsilon_i$ where $\varepsilon_i > 0$ for $i = 1, 2$.  Problem~(MOP$_w$) can be equivalently
written as
\[
\min_{x\in \cal{X}}\
\max\{w\,(F_1(x) - \beta_1^*),\, (1-w)\,(F_2(x) - \beta_2^*)\}\,.
\]

We note that although the objective functionals in Problem~(MOP) are differentiable in their arguments, Problem~(MOP$_w$), or its equivalent above, is still a non-smooth optimization problem, because of the $\max$ operator in the objective. However, it is well-known that Problem~(MOP$_w$) can be re-formulated in the following (smooth) form (cf. the formulation in~\cite{KayNoa2013}) by introducing a new variable $\alpha\ge 0$.
\[
\mbox{(SP$_w$) }\left\{\begin{array}{rl}
\ds\min_{{\alpha\ge0}\atop{x\in \cal{X}}} & \ \ \alpha \\[4mm]
\mbox{subject to} & \ \ w\,(F_1(x) - \beta_1^*)\le\alpha\,, \\[2mm]
& \ \ (1-w)\,(F_2(x) - \beta_2^*) \le\alpha\,.
\end{array} \right.
\]
Problem (SP$_w$) is referred to as {\em goal attainment method}
\cite{Miettinen1999}, as well as {\em Pascoletti--Serafini scalarization} \cite{Eichfelder2008}.  We will solve Problem~(SP$_w$) for the two examples we want to study.

Note that the popular {\em weighted-sum} scalarization, given below, fails to generate the nonconvex parts of a Pareto front.
\[
\mbox{(MOPws)}\qquad \min_{x\in \cal{X}}\ (w\,F_1(x) + (1-w)\,F_2(x))\,.
\]
We will illustrate this with the torus problem in Section~\ref{sec:torus}.  We note that the above weighted-sum scalarization is equivalent to the minimization of \eqref{CIT} (or the Riemannian cubics in tension problem) with $\tau = w/(1-w)$, $w\in[0,1)$.

\subsection{Essential interval of weights}
\label{sec:ess_weights}

With the Chebyshev scalarization, it would usually be enough for the weight $w$ to take values over a (smaller) subinterval $[w_0,w_f]\subset[0,1]$, with $w_0 > 0$ and $w_f<1$, for the generation of the whole front.  Figure~\ref{fig:weights1} illustrates the geometry to compute the subinterval end-points, $w_0$ and $w_f$.  In the illustration, the points $(F_1^*,\Fb_2)$ and $(\Fb_1,F_2^*)$ represent the boundary of the Pareto front.  The equations of the rays which emanate from the utopia point $(\beta_1^*,\beta_2^*)$ and pass through the boundary points are also shown.  By substituting the boundary values of the Pareto curve into the respective equations, and solving each equation for $w_0$ and $w_f$ one simply gets
\begin{equation} \label{w0wf}
w_0 = \frac{(F_2^* - \beta_2^*)}{(\Fb_1 - \beta_1^*) +
  (F_2^* - \beta_2^*)}
\qquad\mbox{and}\qquad
w_f = \frac{(\Fb_2 - \beta_2^*)}{(F_1^* - \beta_1^*) +
  (\Fb_2 - \beta_2^*)}\,.
\end{equation}

\begin{figure}
\begin{center}
\psfrag{pf}{Pareto front}
\psfrag{phi1}{$F_1$}
\psfrag{phi2}{$F_2$}
\psfrag{p1}{$(F_1^*,\Fb_2)$}
\psfrag{p2}{$(\Fb_1,F_2^*)$}
\psfrag{wf}{$w_f\,(F_1 - \beta_1^*) = (1-w_f)\,(F_2 - \beta_2^*)$}
\psfrag{w0}{$w_0\,(F_1 - \beta_1^*) = (1-w_0)\,(F_2 - \beta_2^*)$}
\psfrag{utop}{$(\beta_1^*,\beta_2^*)$}
\includegraphics[width=80mm]{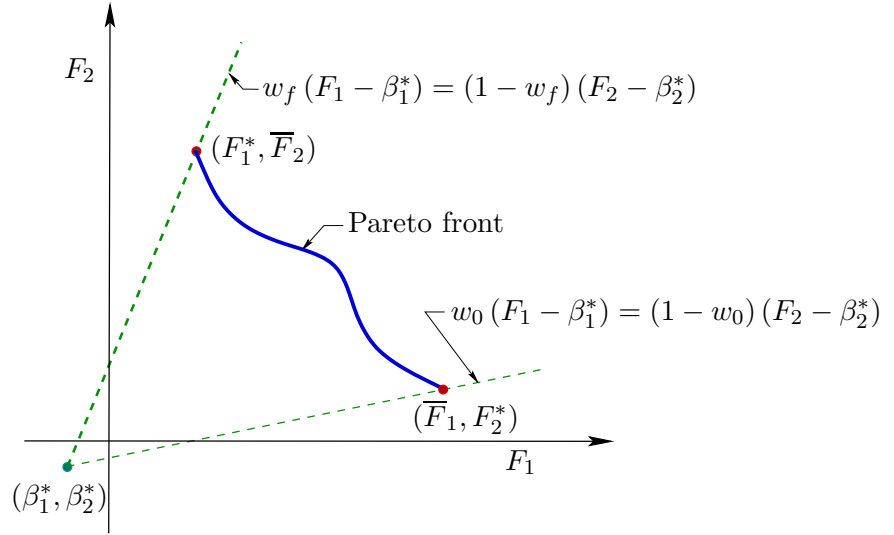}
\caption{\sf Determination of the essential subinterval of weights $[w_0,w_f]$~\cite{KayMau2014}.}
\label{fig:weights1}
\end{center}
\end{figure}

From the geometry depicted in Figure~\ref{fig:weights1}, as also discussed in~\cite{KayMau2014}, one can deduce that with every $w\in[0,w_0]$ the solution of (MOP$_w$) will yield the same boundary point $(\Fb_1,F_2^*)$ on the Pareto front.  Likewise with every $w\in[w_f,1]$ the same boundary point $(F_1^*,\Fb_2)$ is generated.  This observation justifies the avoidance of the weights $w\in[0,w_0)\cup(w_f,1]$ in order not to keep getting the boundary points of the Pareto front, as otherwise one would end up wasting valuable computational effort and time.

%%%%%%%%%%%%%%%%%%%
\section{Convexity Conditions of Pareto Fronts}\label{sec:convex}

The main purpose of this section is to present some necessary and sufficient conditions for Pareto Fronts resulting from the problem (MOP) to be convex. Throughout this paper, a Pareto Front is called convex when its epigraph is convex in the standard sense. In detail, if the Pareto Front $PF$ can be written as $F_2=PF(F_1)$, then $PF$ is convex if and only if $PF(sF_1+(1-s)\tilde{F}_1)\leq sPF(F_1)+(1-s)PF(\tilde{F}_1)$ for any $F_1,\tilde{F}_1$ on the front and $s\in[0,1]$ such that $sF_1+(1-s)\tilde{F}_1$ is also on the front.

Suppose the optimal curve of the problem (MOP) with respect to some scalarization (such as weighted-sum, Chebychev) is $x=x(t,w)$, where $0<w<1$ is the weight parameter. Denote $F_1(w):=\int_{t_0}^{t_f}\langle \dot{x}(t,w),\dot{x}(t,w)\rangle dt$, $F_2(w):=\int_{t_0}^{t_f}\langle \nabla_t \dot{x}(t,w),\nabla_t \dot{x}(t,w)\rangle dt$ and their derivatives with respect to $w$ by $F_1^\prime$ and $F_2^\prime$.

\begin{theorem}
Suppose the Pareto front $PF$ of the problem (MOP) with respect to some weight $w$ can be denoted by $F_2(w)=PF(F_1(w))$ and $F_1,F_2,PF$ are all twice-differentiable. Then, $PF$ is convex if and only if
\begin{align}\label{convex_gen}
F_1^\prime(F_1^\prime F_2^{\prime\prime}-F_1^{\prime\prime}F_2^{\prime})\geq 0.
\end{align}
\end{theorem}

\begin{proof}
    This proof can be seen from the fact that $PF$ is convex if and only if $PF^{\prime\prime}\geq 0$. By straightforward calculations, we have
    \begin{align*}
        F_2^\prime&=PF^\prime F_1^\prime,\\
        F_2^{\prime\prime}&=PF^{\prime\prime} (F_1^\prime)^2+PF^\prime F_1^
{\prime\prime}.
    \end{align*}
    Then substituting the first identity into the second one, we get the relation \eqref{convex_gen}.
\end{proof}

If $F_1^\prime\neq 0$, then we find
\begin{align*}
    PF^{\prime\prime}= \frac{1}{(F_1^{\prime})^2}\left(F_2^{\prime\prime}-\frac{F_2^\prime}{F_1^\prime}F_1^{\prime\prime}\right).
\end{align*}

Now we consider the convexity conditions of Pareto Fronts generated by the problem (MOP) with respect to the weighted-sum scalarization.

\begin{theorem}\label{convex_ws_thm}
    The Pareto front PF of the problem (MOP) generated with respect to the weighted-sum scalarization ($MOP_{ws}$) is convex if and only if
    \begin{align}\label{convex_ws}
        \int_{t_0}^{t_f}\langle \nabla_t\dot{x}(t,w),x^\prime(t,w)\rangle dt\geq 0,
    \end{align}
    where $x$ satisfies the Euler--Lagrange equation \eqref{cubELeq} with $\tau=\frac{w}{1-w}$ $(w\neq 1)$ and $x^\prime$ is the derivative with respect to the weight $w$.
\end{theorem}

\begin{proof}
    By integration by parts, we have
    \begin{align*}
        F_1^\prime &= 2\int_{t_0}^{t_f}\langle \nabla_{x^\prime}\dot{x},\dot{x}\rangle dt=-2\int_{t_0}^{t_f}\langle\nabla_{t}\dot{x},x^\prime\rangle dt,\\
        F_2^\prime &= 2\int_{t_0}^{t_f}\langle \nabla_{x^\prime}\nabla_t \dot{x},\nabla_t\dot{x}\rangle dt=2\int_{t_0}^{t_f}\langle \nabla_t^2 x^\prime+R(x^\prime,\dot{x})\dot{x},\nabla_t\dot{x}\rangle dt\\
        &= 2\int_{t_0}^{t_f}\langle \nabla_t^3\dot{x}+R(\nabla_t\dot{x},\dot{x})\dot{x},x^\prime\rangle dt = 2\tau\int_{t_0}^{t_f}\langle \nabla_t\dot{x},x^\prime\rangle dt=-\tau F_1^\prime,
    \end{align*}
    where we have used the Euler--Lagrange equation \eqref{cubELeq} in the second last identity above. Then, condition \eqref{convex_gen} implies
    \begin{align*}
        F_1^\prime(F_1^\prime F_2^{\prime\prime}-F_1^{\prime\prime} F_2^\prime)=-\frac{(F_1^\prime)^3}{(1-w)^2}\geq 0\implies F_1^\prime \leq 0\implies \int_{t_0}^{t_f}\langle \nabla_t\dot{x}(t,w),x^\prime(t,w)\rangle dt\geq 0,
    \end{align*}
    which completes this proof.
\end{proof}

Frankly speaking, the condition \eqref{convex_ws} is not easy to verify in practice. Partly because it is difficult to solve the the Euler--Lagrange equation \eqref{cubELeq} to get $x(t,w)$ and partly because evaluating the derivative of $x(t,w)$ with respect to $w$ and its associated integration brings a lot of difficulties. Now we apply Theorem \ref{convex_ws_thm} to the case where the Riemannian manifold $M$ is the standard Euclidean space, i.e., the Euler--Lagrange equation \eqref{cubELeq} can be solved explicitly.

\begin{remark}\label{cor_convex_eculidean}
    When $M$ is specified as the Euclidean space $\mathbb{E}^n$, the Pareto front PF of the problem (MOP) generated with respect to the weighted-sum scalarization ($MOP_{ws}$) is convex for the weight $w$ on some interval $(w_0,1)$, where $0<w_0<1$. Refer to Appendix \ref{cor_convex_pf} for more detailed discussions.
\end{remark}

The following example gives a visualisation of the Pareto front of the multi-objective optimisation problem in the Euclidean space and verifies Theorem \ref{convex_ws_thm}.

\begin{example}
    Choosing $T=1$, $x_0=(0,0,0)^T$, $v_0=(1,0,0)^T$, $x_T=(2,2,1)^T$, $v_T=(0,1,-1)^T$, we can get the analytical form of cubics in tension in $\mathbb{E}^3$ (see expressions in Appendix \ref{cor_convex_pf}). Figure \ref{fig:Euclidean_PF} shows the Pareto front of the problem (MOP) with respect to the scalarization (MOP$_{ws}$) and plots the convexity condition \eqref{convex_ws}.
    
    \begin{figure}[ht]
    \centering
    \includegraphics[width=75mm]{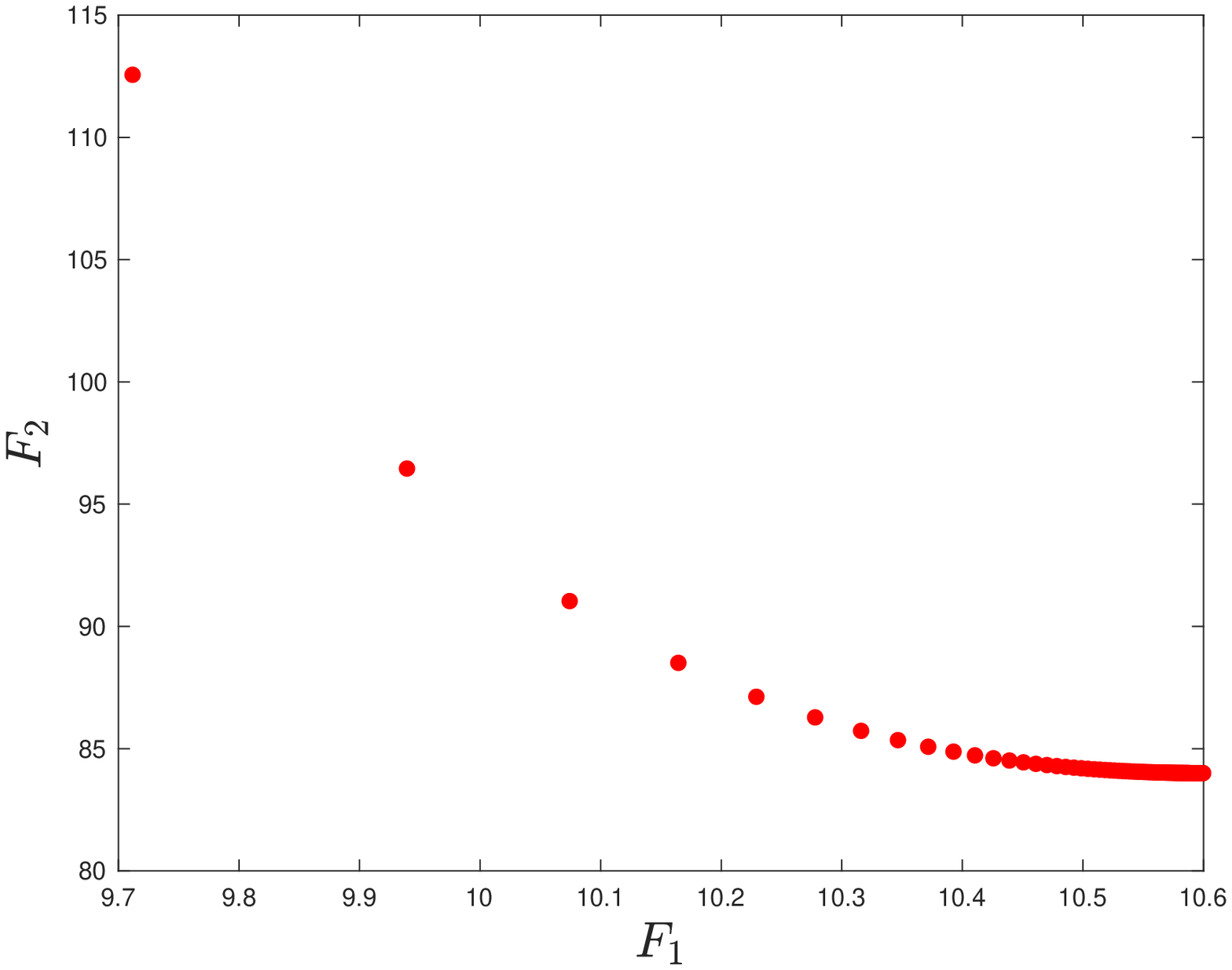}
    \includegraphics[width=75mm]{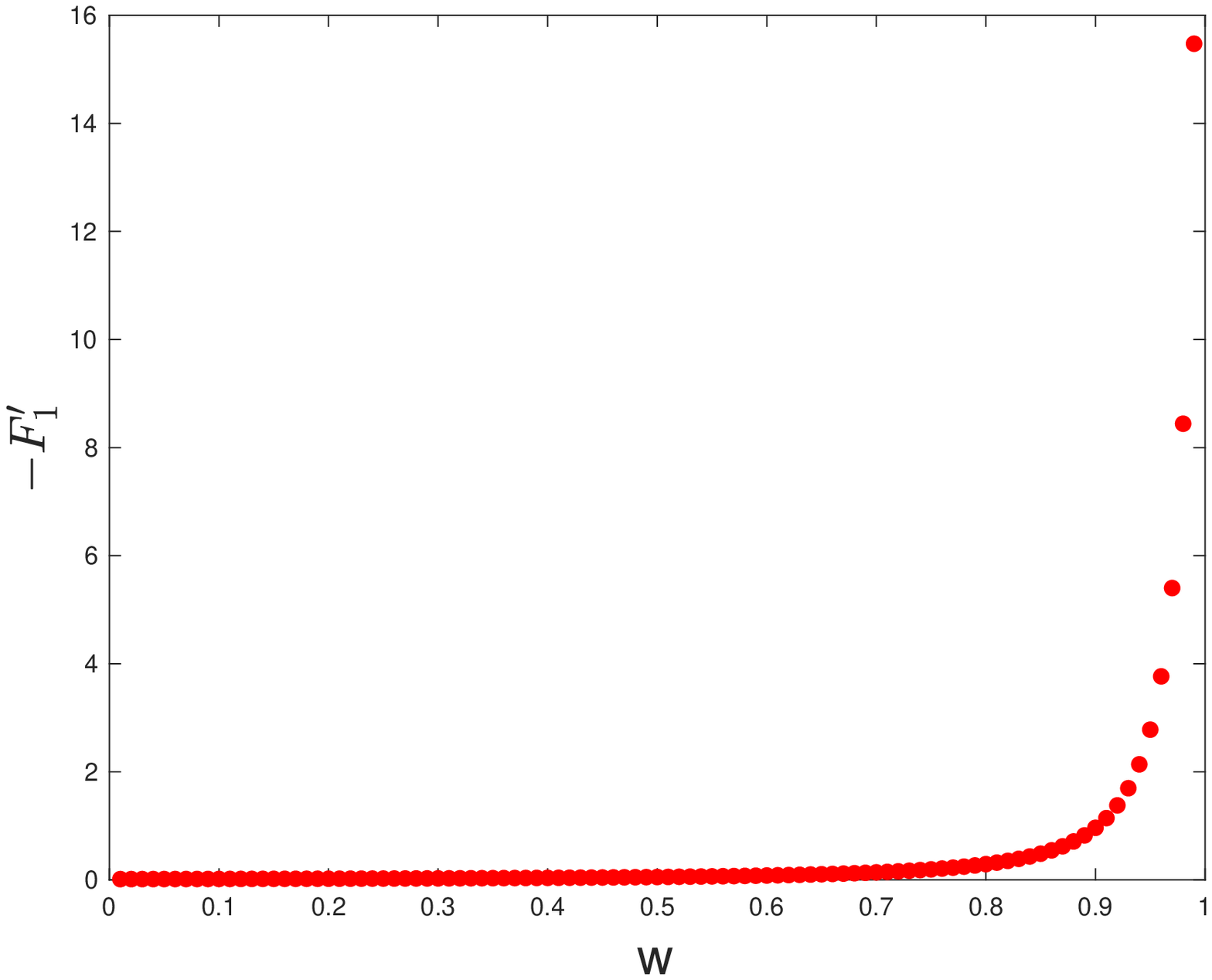} \\[5mm]
    \caption{\sf (Left) Pareto front (Right) Convexity condition}\label{fig:Euclidean_PF}
    \end{figure}
\end{example}

\section{Numerical Experiments}\label{sec:num}

For numerical experiments, we consider variational curves on two two-dimensional manifolds, namely on sphere and torus.  To construct an approximation of the Pareto front of a given problem we implement Algorithm~1 in~\cite{KayMau2014}.

We employ the scalarize--discretize--then--optimize approach that was previously used in~\cite{KayMau2014}.  Under this approach, one first scalarizes the multi-objective problem in the infinite-dimensional space, and then discretizes the scalarized problem directly (in time space) and applies a usually large-scale finite-dimensional optimization method to find a discrete approximate solution of the scalarized problem. By the existing theory of discretization (see for example \cite{AltBaiLemGer2013, Betts2020, DonHag2001, DonHagMal2000, DonHagVel2000, PieScaVel2018}), under some general assumptions, the discrete approximate solution converges to a solution of the continuous-time scalarization~(SP$_w$) of the original problem, yielding a Pareto minimum of the original problem.

In the subproblems of the algorithm in~\cite{KayMau2014}, a direct discretization of Problem~(SP$_w$) employing the trapezoidal rule is solved by using Knitro, version 13.0.1.  Knitro is a popular optimization software; see~\cite{Knitro}. We use AMPL \cite{AMPL} as an optimization modelling language, which employs Knitro as a solver.  We set the Knitro parameters {\tt feastol=1e-8} and {\tt opttol=1e-8}.  We also set the number of discretization points to be 2,000 for the sphere and 5,000 for the torus.

\subsection{Sphere}

Let $\mathbb{S}^2:=\{x=(x_1,x_2,x_3)\in \mathbb{E}^3~\vert~ \Vert x\Vert=1\}$ be the unit-sphere endowed with the standard Euclidean metric/norm. With respect to the induced metric, the covariant derivative $\nabla_t\dot{x}(t)$ on $\mathbb{S}^2$ can be written as
\begin{align}
\nabla_t\dot{x}(t)=\ddot{x}(t)-\langle\ddot{x}(t),x(t)\rangle x(t)=\ddot{x}(t)+\langle\dot{x}(t),\dot{x}(t)\rangle x(t),
\end{align}
where we have used the relation of the twice differentiation of the constraint $\langle x(t),x(t)\rangle=1$. Then the squared norm of the covariant acceleration on $\mathbb{S}^2$ is given by 
\begin{align}
\Vert \nabla_t\dot{x}(t)\Vert^2=\Vert \ddot{x}(t)\Vert^2-\Vert \dot{x}(t)\Vert^4.
\end{align}

Now the general multi-objective optimization problem (MOP) on $\mathbb{S}^2$ can be specified as follows,
\[
\mbox{(S) }\left\{\begin{array}{rl}
\ds\min_{x(\cdot)} & \ \ \ds\left[\int_{t_0}^{t_f} \|\dot{x}(t)\|^2\,dt\,,\ \ \int_{t_0}^{t_f} \left(\|\ddot{x}(t)\|^2 - \|\dot{x}(t)\|^4\right)\,dt\right]
                          \\[4mm] 
\mbox{subject to} & \ \ x_1(t)^2 + x_2(t)^2 + x_3(t)^2 = 1\,, \mbox{\ \ for all } t\in[0,t_f]\,, \\[2mm]
& \ \ x(t_0) = x_0\,\ \ \dot{x}(t_0) = v_0\,,\ \ x(t_f) = x_f\,\ \ \dot{x}(t_f) = v_f\,.
\end{array} \right.
\]
We have
\[
F_1(x) := \int_{t_0}^{t_f} \|\dot{x}(t)\|^2\,dt \qquad\mbox{and}\qquad
F_2(x) := \int_{t_0}^{t_f} \left(\|\ddot{x}(t)\|^2 - \|\dot{x}(t)\|^4\right)\,dt\,.
\]
Then the smooth version of the Chebyshev scalarization of the sphere problem~(S) can simply be stated as
\[
\mbox{(S$_w$) }\left\{\begin{array}{rl}
\ds\min_{x(\cdot),\,\alpha} & \ \ \alpha \\[4mm] 
\mbox{subject to} & \ \ w\,F_1(x) \le\alpha\,, \\[2mm]
& \ \ (1-w)\,F_2(x) \le\alpha\,, \\[2mm]
&\ \ x_1(t)^2 + x_2(t)^2 + x_3(t)^2 = 1\,, \mbox{\ \ for all } t\in[0,t_f]\,, \\[2mm]
& \ \ x(t_0) = x_0\,\ \ \dot{x}(t_0) = v_0\,,\ \ x(t_f) = x_f\,\ \ \dot{x}(t_f) = v_f\,,
\end{array} \right.
\]
where $w\in[0,1)$ is fixed and we have taken $\beta_1^* = \beta_2^* = 0$.

%Define the new control variable $v(t) := \dot{u}(t)$ for
%a.e. $t\in[0,t_f]$, with $v(t) = (v_1(t),\ldots,v_m(t))\in\dR^m$.
%Then Problem~(P) can be augmented and thus transformed to
%\[
%\mbox{(Pa) }\left\{\begin{array}{rl}
%\ds\min & \ \ \ds\int_0^{t_f} \left(f_0(x(t),u(t),t) +
%          \alpha\,\|v(t)\|_1\right) dt  \\[4mm] 
%\mbox{subject to} & \ \ \dot{x}(t) = f(x(t),u(t),t)\,, \mbox{\ \ for
%  a.e. } t\in[0,t_f]\,,\ \ x(0) = x_0\,,\ \ x(t_f) = x_f\,, \\[2mm]
%& \ \ \dot{u}(t) = v(t)\,,\mbox{\ \ for a.e. } t\in[0,t_f]\,,\ \ -1\le
%  u(t)\le 1\,,
%\end{array} \right.
%\]
%where $\|v(t)\|_1 = |v_1(t)| + \ldots + |v_m(t)|$.  

\begin{figure}
\begin{center}
\includegraphics[width=140mm]{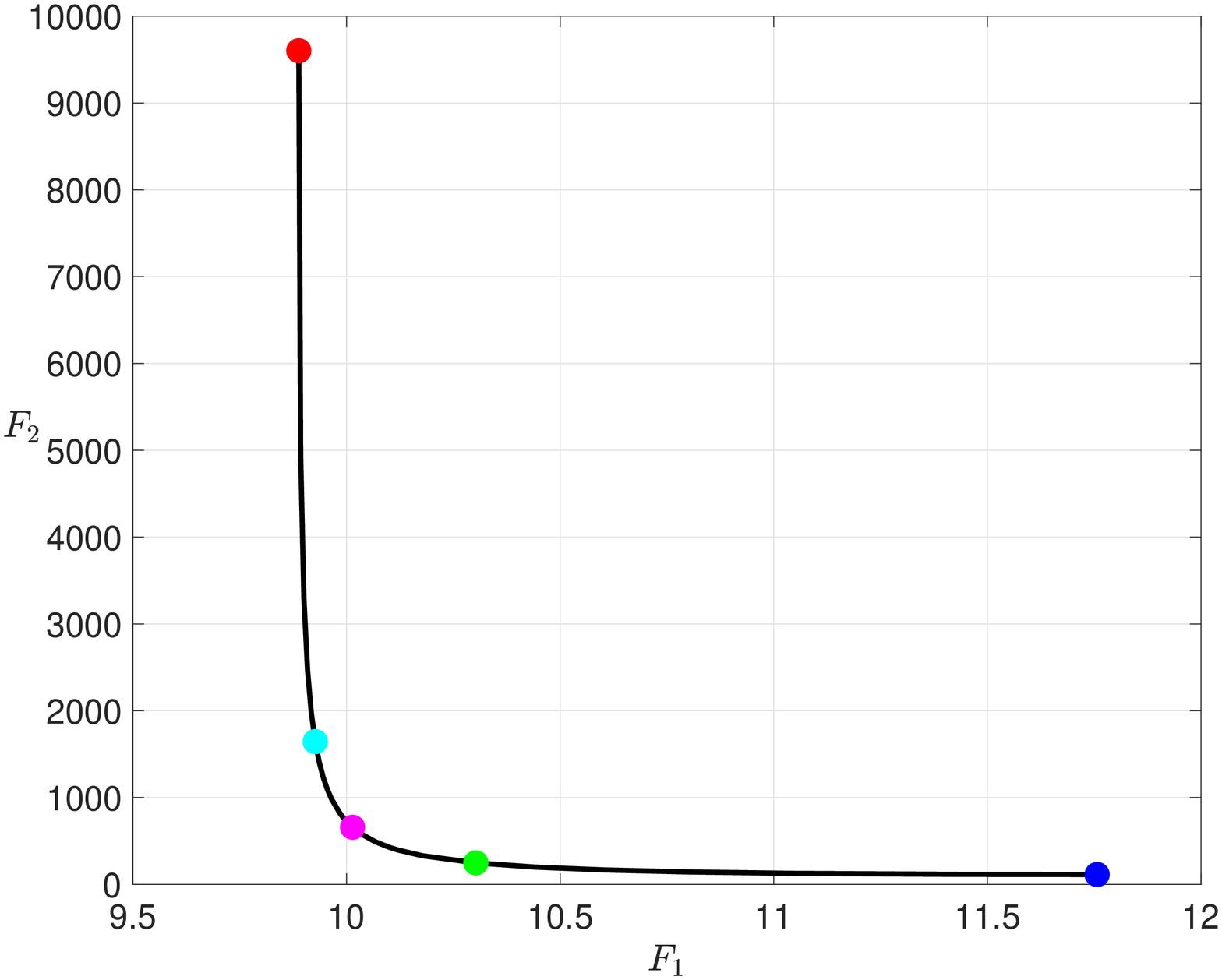} \\[5mm]
\end{center}
\begin{minipage}{90mm}
\hspace*{-15mm}
\includegraphics[width=95mm]{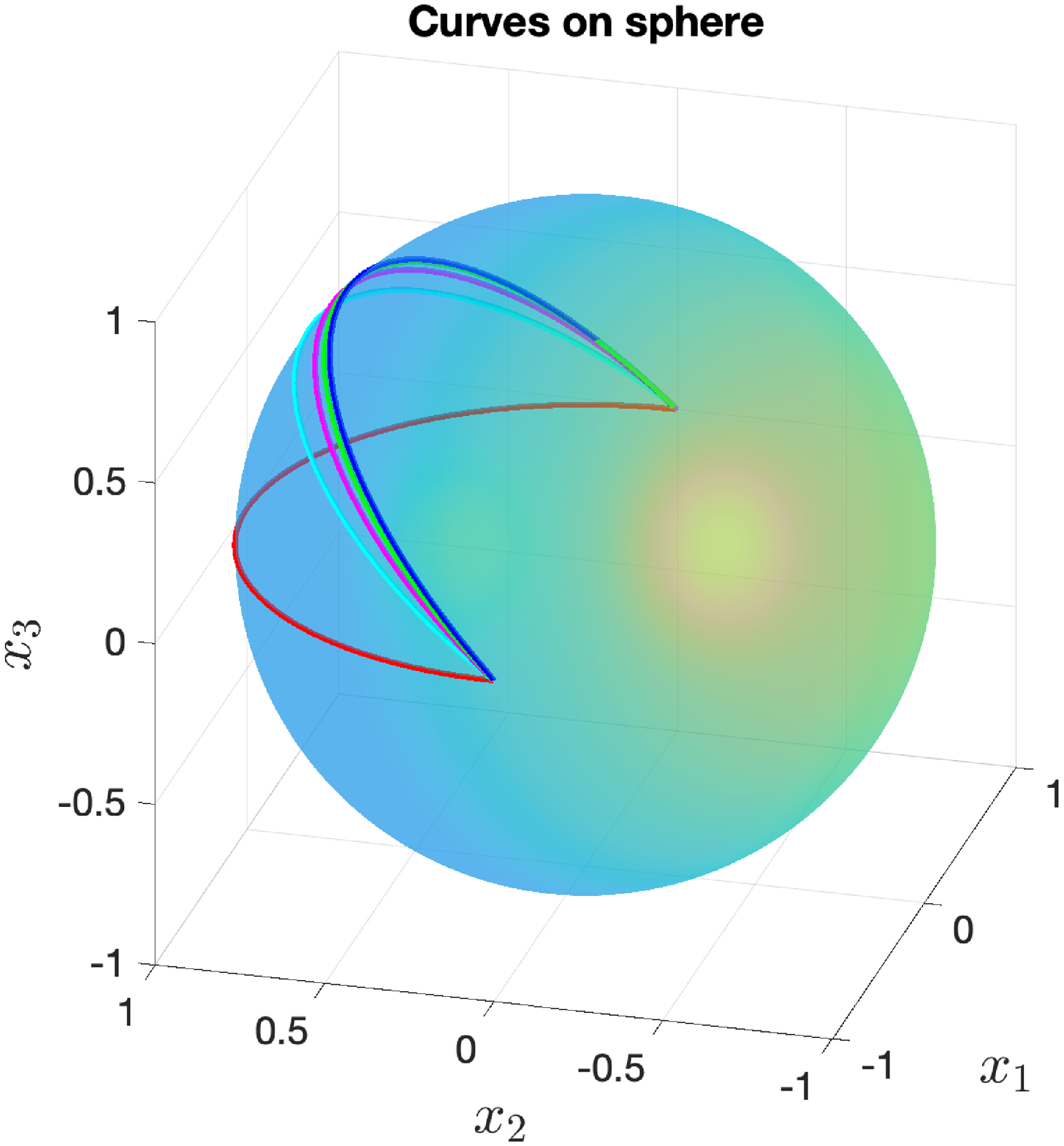} \\[3mm]
\end{minipage}
\begin{minipage}{80mm}
\begin{minipage}{80mm}
\hspace*{-20mm}
\includegraphics[width=90mm]{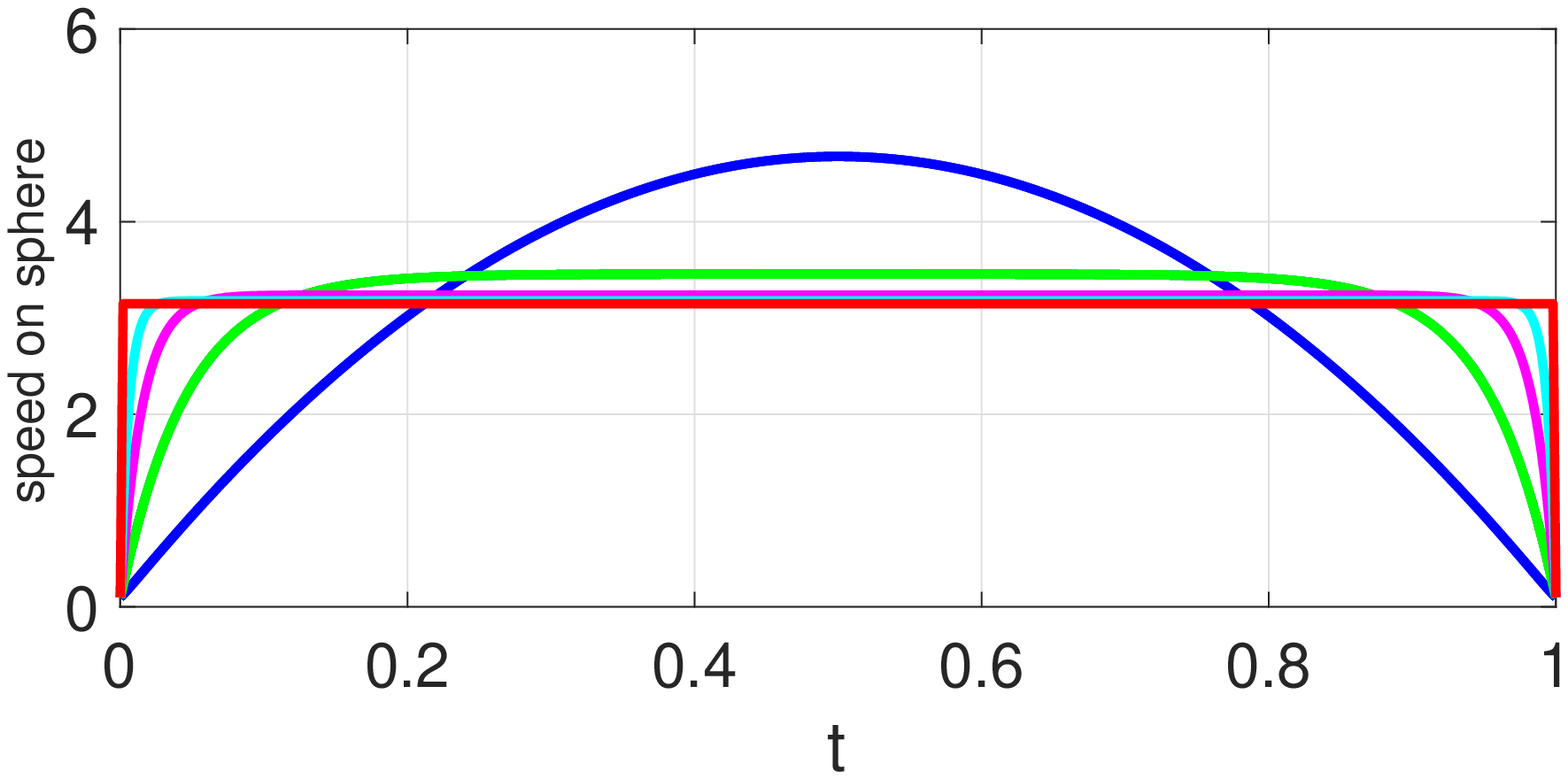} \\[3mm]
\end{minipage}
\\
\begin{minipage}{80mm}
\hspace*{-20mm}
\includegraphics[width=90mm]{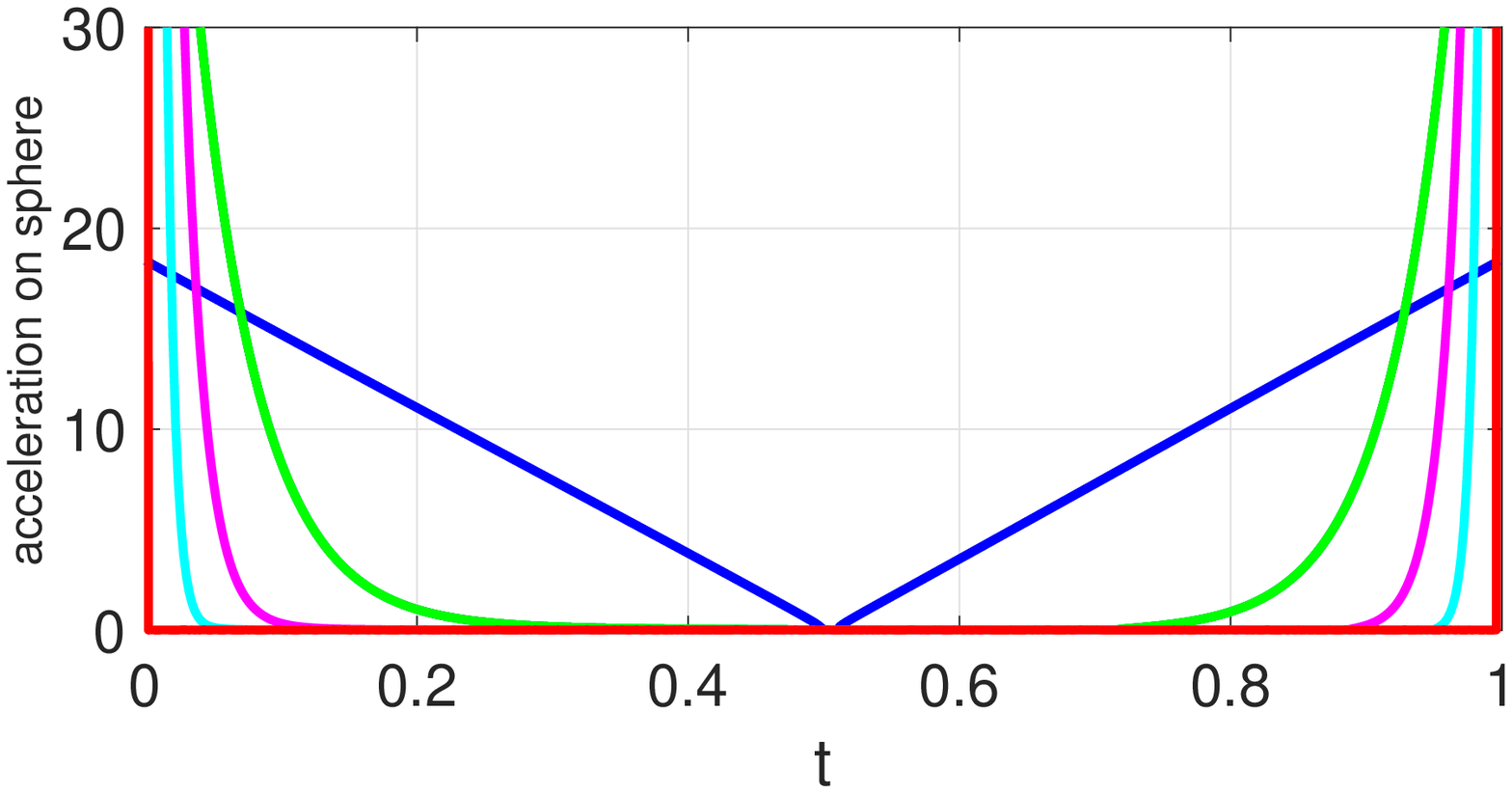} \\[3mm]
\end{minipage}
\end{minipage}
\
\caption{\sf The Pareto front, the speed and acceleration on the sphere for a Riemannian cubic in tension, $x:[0,1]\to \mathbb{S}^2$, from $(x_0,v_0) = ((1,0,0),(0,0.1,0))$ to $(x_f,v_f) = ((-1,0,0),(0,0,-0.1))$.} 
\label{fig:Pareto}
\end{figure}

In Figure~\ref{fig:Pareto}, ``speed on sphere'' is simply\ \ $\|\dot{x}(t)\|$\ \ and the ``acceleration on sphere" is \ \ $\sqrt{\|\ddot{x}(t)\|^2 - \|\dot{x}(t)\|^4}$\,.

In generating the Pareto solutions and an approximation of the Pareto front, we have solved (S$_w$) in the subproblems of Algorithm~1 of ~\cite{KayMau2014},  with $0\le w < 1$.

The colour-coded Pareto solutions in Figure~\ref{fig:Pareto} are given as blue ($0\le w\le0.9059$), green ($w = 0.96$), magenta ($w = 0.985$), cyan ($w = 0.9940$), red ($w \to 1^-$).  The meaningful interval of the $w$ values is $[0.9059,1)$.  The solution denoted by blue is nothing but a (single) Riemannian cubic with $(F_1,F_2) = (11.76, 113.2)$.

Numerical observations suggest that the solution represented by red is asymptotically given by $(F_1,F_2) = (9.88, \infty)$, with $F_1$ correct to three decimal places, as $w \to 1^-$. For practical (illustration) purposes we depict $F_2 < \infty$ in Figure~\ref{fig:Pareto}; otherwise the Pareto curve extends to infinity to the left.

The curves on the sphere appear as if they have different end-velocities, which is not the case except for the red curve.  The end-accelerations of the solution coded in red seem to be impulsive, resulting in a jump in the end-speeds.  This can be observed as instantaneous rotation of the end-velocities of the curves on the sphere.

The next example illustrates that not only the Pareto front can be disconnected but also can contain two distinct Riemannian cubics.

\subsection{Torus}
\label{sec:torus}

Torus $\mathbb{T}^2$ is a geometric object described by 
\begin{eqnarray*}
\big\{x=(x_1,x_2,x_3)\in\dR^3\ &|\ & x_1 = (c+a\cos v)\cos u,\ x_2 = (c+a\cos v)\sin u,\ x_3 = a\sin v,  \\
&& u,v\in[0,2\pi) \big\}\,
\end{eqnarray*}
with the induced standard Euclidean metric, where the constants $a$ and $c$ are respectively the smaller and greater radii of the torus.

By straightforward calculations, we get the derivatives
\begin{align*}
    x_u =& (-(c+a\cos v)\sin u,(c+a \cos v)\cos u,0),\\
    x_v =& (-a\sin v\cos u,-a\sin v\sin u, a\cos v).
\end{align*}
and the first fundamental form
\begin{align*}
    ds^2 &=\langle x_u,x_u\rangle du^2+2\langle x_u,x_v\rangle dudv+\langle x_v,x_v\rangle dv^2\\
    &= (c+a\cos v)^2du^2+a^2dv^2,
\end{align*}
which implies the Riemannian metric $g$ and its inverse $g^{-1}$ as follows,
\begin{align*}
    g=\begin{bmatrix}
    (c+a\cos v)^2 & 0\\
    0 & a^2
    \end{bmatrix},~~
    g^{-1}=\begin{bmatrix}
    (c+a\cos v)^{-2} & 0\\
    0 & a^{-2}
    \end{bmatrix}.
\end{align*}
Further, the non-zero Christoffel symbols are given by
\begin{align*}
    \Gamma_{uu}^v&=-\frac{1}{2}g^{vv}\frac{g_{uu}}{\partial v}=\left(\frac{c}{a}+\cos v\right)\sin v,\\
    \Gamma_{uv}^u&=\Gamma_{vu}^u=\frac{1}{2}g^{uu}\frac{\partial g_{uu}}{\partial v}=-\frac{a\sin v}{c+a\cos v}.
\end{align*}
Therefore, the covariant derivative $\nabla_t\dot{x}(t)=\ddot{x}_i(t)+\Gamma_{jk}^i\dot{x}_j(t)\dot{x}_k(t)$ on the torus $\mathbb{T}^2$ can be denoted by
\[
\nabla_t\dot{x}(t) = (a_1(t),a_2(t))\,,
\]
where
\begin{align*}
a_1 &= \ddot{u}+\Gamma_{uv}^u\dot{u}\dot{v}+\Gamma_{vu}\dot{v}\dot{u}=\ddot{u} - \frac{2a\sin v}{c + a\cos v}\,\dot{u}\dot{v},\\
a_2 &= \ddot{v}+\Gamma_{uu}^v\dot{u}\dot{u}=\ddot{v} + \frac{1}{a} \sin v (c + a\cos v)\,\dot{u}^2.
\end{align*}

The problem of interest is the simultaneous minimization of the objective functionals
\[
F_1 := \int_0^{t_f}\left(\dot{x}_1^2(t) + \dot{x}_3^2(t) + \dot{x}_3^2(t)\right) dt \quad\mbox{and}\quad
F_2 := \int_0^{t_f}\left(a_1^2(t) + a_2^2(t)\right) dt\,.
\]
In other words, we are interested in solving the problem
\[
\mbox{(T) }\left\{\begin{array}{rl}
\ds\min_{u(\cdot),v(\cdot)} & \ \ \ds\left[F_1\,,\ F_2\right]
                          \\[4mm] 
\mbox{subject to} & \ \ (u(0), v(0)) = (u_0,v_0)\,,\ \ (u(t_f), v(t_f)) = (u_f,v_f)\,, \\[1mm]
& \ \ (\dot{u}(0), \dot{v}(0)) = (\dot{u}_0,\dot{v}_0)\,,\ \ (\dot{u}(t_f), \dot{v}(t_f)) = (\dot{u}_f,\dot{v}_f)\,.
\end{array} \right.
\]

Recall that a certain trade-off solution set of Problem~(T) is referred to as the Pareto set.

The Chebyshev scalarization of Problem~(T) is given by
\[
\mbox{(T$_w$) }\left\{\begin{array}{rl}
\ds\min_{u(\cdot),v(\cdot),\alpha} & \ \ \alpha   \\[4mm] 
\mbox{subject to} & \ \ w\,F_1(x) \le\alpha\,, \\[2mm]
& \ \ (1-w)\,F_2(x) \le\alpha\,, \\[2mm]
& \ \ (u(0), v(0)) = (u_0,v_0)\,,\ \ (u(t_f), v(t_f)) = (u_f,v_f)\,, \\[1mm]
& \ \ (\dot{u}(0), \dot{v}(0)) = (\dot{u}_0,\dot{v}_0)\,,\ \ (\dot{u}(t_f), \dot{v}(t_f)) = (\dot{u}_f,\dot{v}_f)\,.
\end{array} \right.
\]

Take an example instance: Let $u_0 = v_0 = 0$, $u_f = \pi$, $ v_f = \pi/2$.  Let $\dot{u}_0 = 1$, $\dot{v}_0 = 0$, $\dot{u}_f = 0$, $\dot{v}_f = -1$, with $t_f = 1$.  See Figure~\ref{fig:Pareto_torus} for the Pareto front constructed by using the scalarization (T$_w$), after weeding out the dominated points found as local minima of (T$_w$).  The essential interval of the weights is found to be $[0.593, 1)$.  The right-most (boundary) point in the front represents the Riemannian cubic found with $w = w_0 = 0.5931$ as $(F_1,F_2) = (170.143, 248.049)$, indicated with the colour green.  As can be seen the Pareto front is disconnected with a large gap.  Moreover, the right-hand segment's epigraph is nonconvex, prohibiting the use of the classical weighted sum as scalarization.

Numerical observations suggest that the left-hand part of the Pareto front is asymptotic to $(F_1,F_2) = (80.2, \infty)$, correct to three significant figures, as $w\to 1^-$.  The related curves are shown in the colour red.  Moreover, the Pareto solution $(F_1,F_2) = (95.15, 457.2)$, found with $w = 0.8259$ and indicated with the colour blue, represents another Riemannian cubic.  To the best knowledge of the authors this is the first nontrivial example of multiple Riemannian cubics with the same boundary data.

The remaining boundary point is indicated with the colour light blue, computed as $(F_1,F_2) = (138.8, 456.3)$ with $w = 0.7667$.

\begin{figure}
\begin{center}
\includegraphics[width=140mm]{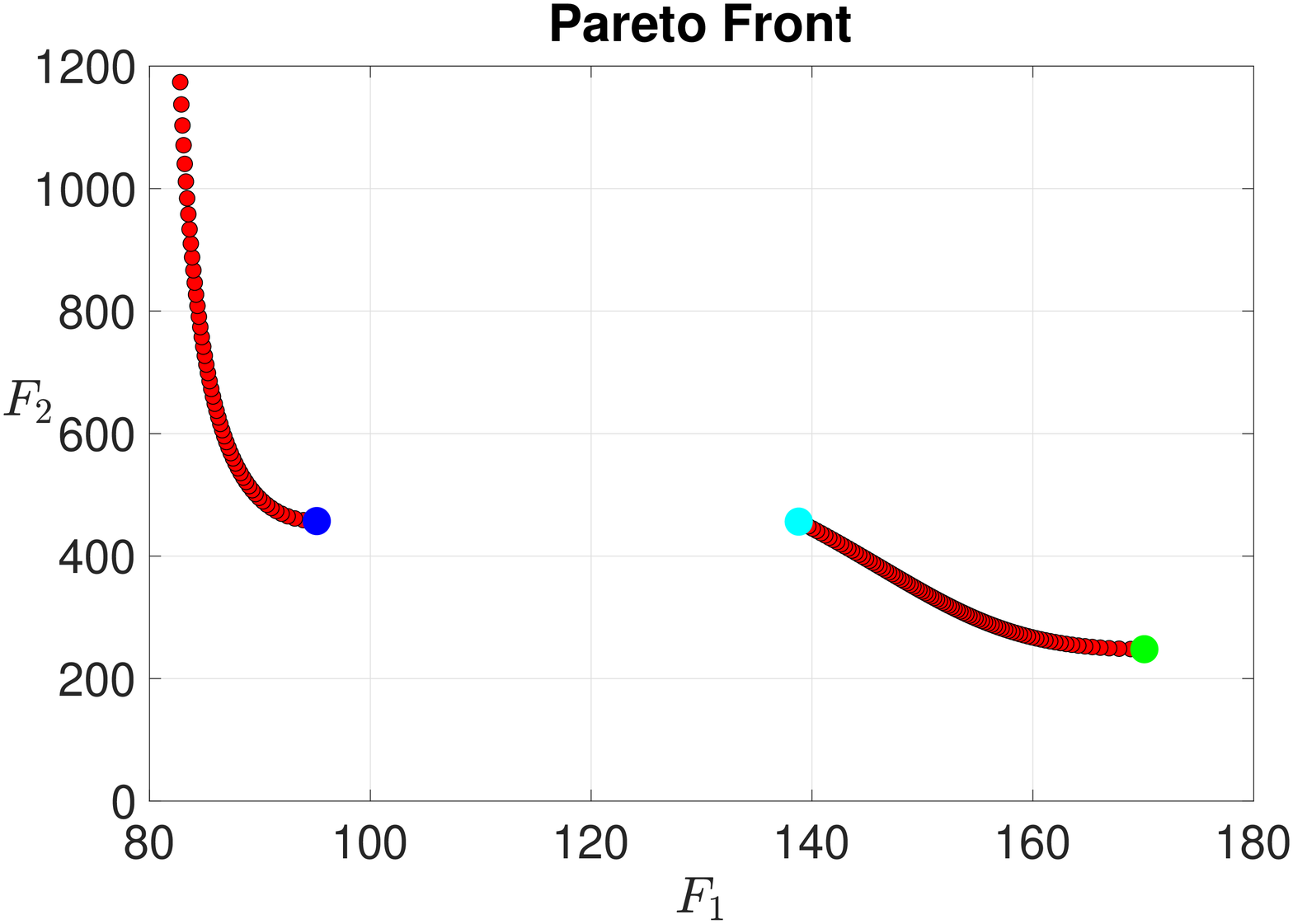}
\end{center}
\begin{minipage}{100mm}
\hspace*{-10mm}
\psfrag{x}{$x_1$}
\includegraphics[width=90mm]{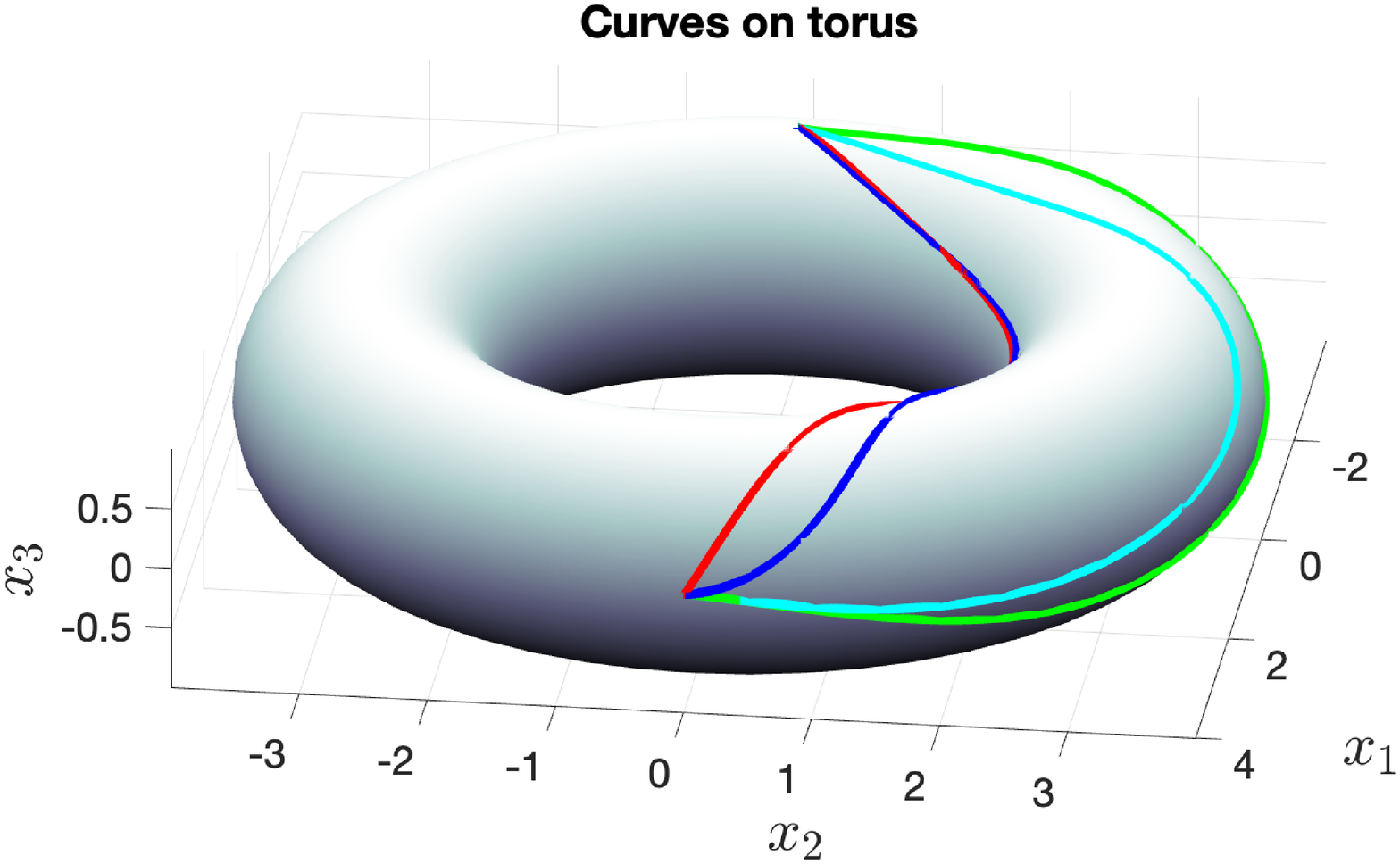}
\end{minipage}
\begin{minipage}{80mm}
\begin{minipage}{80mm}
\hspace*{-20mm}
\includegraphics[width=80mm]{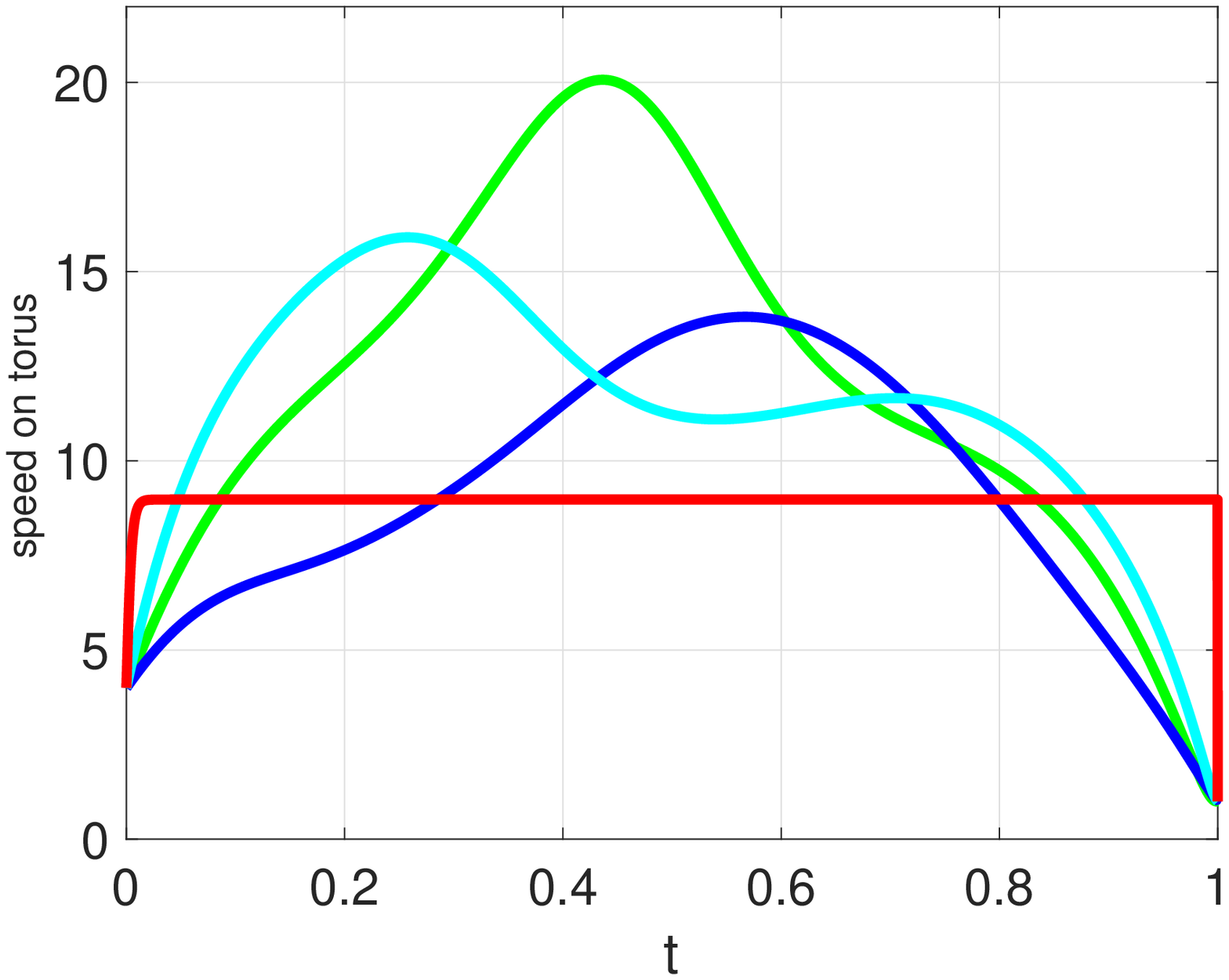}
\end{minipage}
\\
\begin{minipage}{80mm}
\hspace*{-20mm}
\includegraphics[width=80mm]{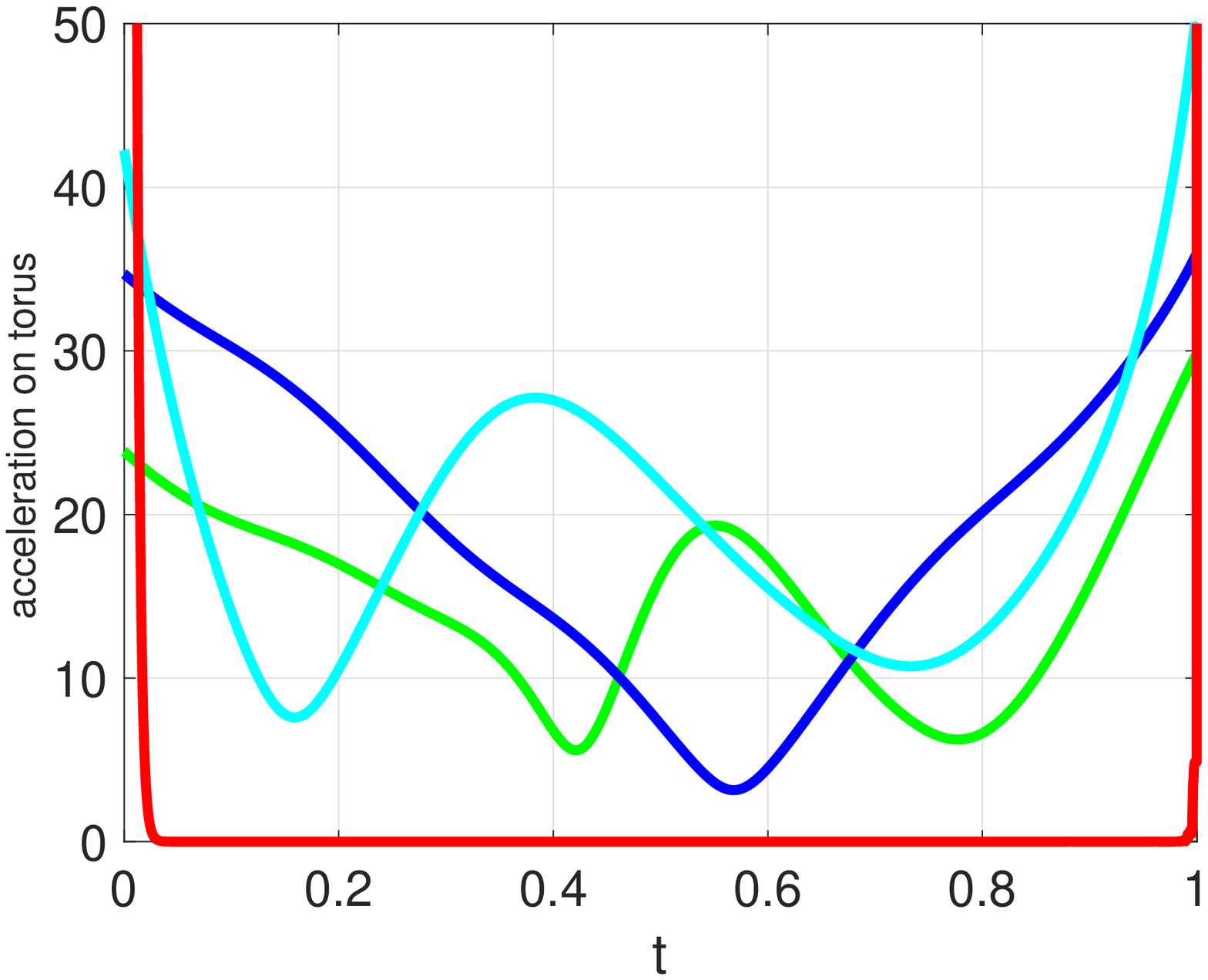}
\end{minipage}
\end{minipage}
\caption{\sf Torus---Pareto front constructed via~(T$_w$).  Curves are obtained between the oriented points $(u_0,v_0,\dot{u}_0,\dot{v}_0) = (0,0,1,0)$ and $(u_f,v_f,\dot{u}_f,\dot{v}_f) = (\pi,\pi/2,0,-1)$.}
\label{fig:Pareto_torus}
\end{figure}

The weighted-sum scalarization of Problem~(T) (or the classical expression for the Riemannian cubics in tension) is given by
\[
\mbox{(Tws) }\left\{\begin{array}{rl}
\ds\min_{u(\cdot),v(\cdot)} & \ \ w\,F_1 + (1-w) F_2  \\[4mm] 
\mbox{subject to} & \ \ (u(0), v(0)) = (u_0,v_0)\,,\ \ (u(t_f), v(t_f)) = (u_f,v_f)\,, \\[1mm]
& \ \ (\dot{u}(0), \dot{v}(0)) = (\dot{u}_0,\dot{v}_0)\,,\ \ (\dot{u}(t_f), \dot{v}(t_f)) = (\dot{u}_f,\dot{v}_f)\,.
\end{array} \right.
\]

See Figure~\ref{fig:Pareto_torus_ws} for the Pareto front constructed by using the scalarization (Tws).  When compared with the Pareto front in Figure~\ref{fig:Pareto_torus} the classical expression for the Riemannian cubics in tension clearly misses an important segment of compromise/Pareto solutions.

We note that the Euler--Lagrange equations as optimality conditions of Problem~(Tws) is provided in Appendix~\ref{app:B}.

\begin{figure}
\begin{center}
\includegraphics[width=140mm]{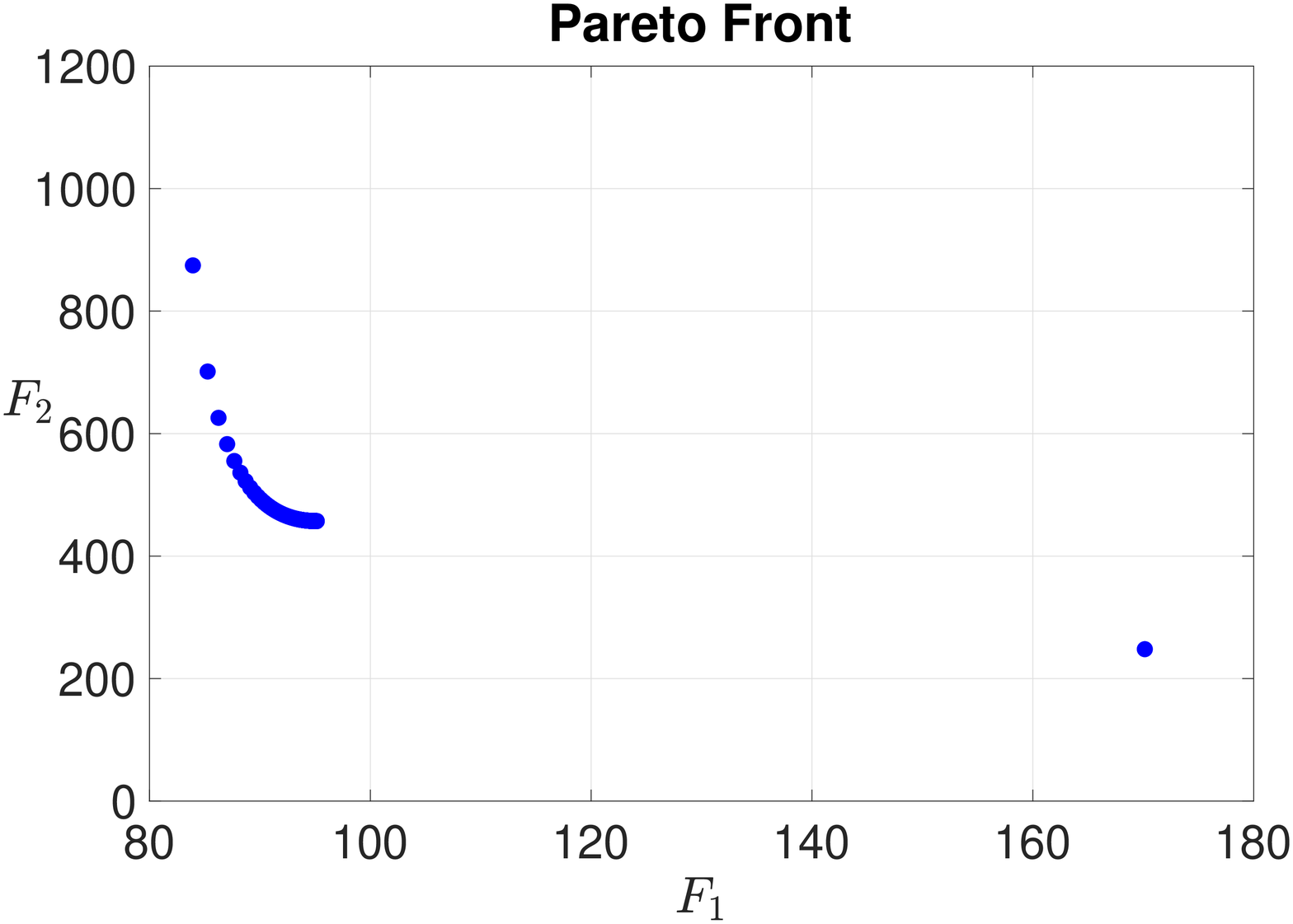} \\[5mm]
\end{center}
\caption{\sf Torus---Pareto front constructed via~(Tws).} 
\label{fig:Pareto_torus_ws}
\end{figure}

\section{Conclusion}\label{sec:con}

Measurements from variational curves, for example, kinetic energy and squared norm of higher-order derivatives, are seldom minimized at the same time in the literature. Motivated by the problem of finding Riemannian cubics in tension, we have formulated a multi-objective optimisation problem, that is, we have posed the minimization of the vector of total kinetic energy and squared norm of acceleration of a curve on a Riemannian manifold. To obtain a sequence of single-objective optimization problems, we adopted the Chebychev and weighted-sum scalarizations in this paper. Some numerical experiments presented on a sphere and a torus illustrated disconnected Pareto fronts and non-uniqueness of Riemannian cubics. To the authors' best knowledge, these are non-trivial observations that may bring the optimization and differential geometry communities' attention. Furthermore, we derived some convexity conditions for the Pareto fronts on general Riemannian manifolds. 

Future research may include more theoretical studies on the phenomenon of non-unique Riemannian cubics on a torus and the disconnected Pareto fronts that they belong to.  We have investigated the local optimality of the two Riemannian cubics by utilizing the test looking for bounded solutions of certain Riccati equations presented in~\cite{MauPic1995}.  However, it was not possible for us to obtain bounded Riccati solutions deeming the test inconclusive.  Further investigation in search for bounded Riccati solutions might also be a direction interesting to pursue.

\ \\
\noindent
{\bf\Large Data Availability} \\[2mm]
The full resolution Matlab graph/plot files that support the findings of this study are available from the corresponding author upon request.

\ \\
\noindent
{\bf\Large Conflict of Interest} \\[2mm]
The authors have no competing, or conflict of, interests to declare that are relevant to the content of this article.

\appendix
\renewcommand{\appendixname}{Appendix~\Alph{section}}
%%%%%%%%%%%%%%%%%%
\section{Discussion on Remark \ref{cor_convex_eculidean}}\label{cor_convex_pf}

Cubics in tension are solutions of the following Euler--Lagrange equation
\begin{align*}
    \begin{cases}
        x^{(4)}(t)-\tau \ddot{x}(t)=0,\\
        x(0)=x_0,\dot{x}(0)=v_0,x(T)=x_T,\dot{x}(T)=v_T,
    \end{cases}
\end{align*}
which can be analytically written as
\begin{align*}
    x=x_0+v_0 t+(\sinh(\sqrt{\tau}t)-\sqrt{\tau}t)c_1+(\cosh(\sqrt{\tau}t)-1)c_2,
\end{align*}
where $c_1$ and $c_2$ satisfy
\begin{align*}
    \begin{pmatrix}
        c_1\\
        c_2
    \end{pmatrix}=\frac{1}{\Delta}\begin{pmatrix}
    \sqrt{\tau}T\sinh(\sqrt{\tau}T) & 1-\cosh(\sqrt{\tau}T)\\
    \sqrt{\tau}T(1-\cosh(\sqrt{\tau}T)) & \sinh(\sqrt{\tau}T)-\sqrt{\tau}T
    \end{pmatrix}\begin{pmatrix}
     x_T-x_0-v_0T\\
     (v_T-v_0)T
    \end{pmatrix}
\end{align*}
with $\Delta=\sqrt{\tau}T(2\cosh(\sqrt{\tau}T)-2-\sqrt{\tau}T\sinh(\sqrt{\tau}T))$. 

To simplify notations, we denote by $c=(c_1,c_2)^T$, $A=\begin{pmatrix}\sinh(\tau_T)-\tau_T & \cosh(\tau_T)-1\\\tau_T(\cosh(\tau_T)-1) & \tau_T\sinh(\tau_T)\end{pmatrix}$, $b=(x_T-x_0-v_0T,(v_T-v_0)T)^T$, $\tau_T=\sqrt{\tau}T$, $\tau_t=\sqrt{\tau}t$. Then, $Ac=b$, which implies $c^\prime=-A^{-1}A^\prime c$.

Straightforward calculations yield
\begin{align*}
    \ddot{x}&=\tau(\sinh(\tau_t),\cosh(\tau_t))c,\\
    x^\prime &=\frac{\tau^\prime t}{2\sqrt{\tau}}(\cosh(\tau_t)-1,\sinh(\tau_t))c-(\sinh(\tau_t)-\tau_t,\cosh(\tau_t)-1)A^{-1}A^\prime c.
\end{align*}
Note that
\begin{align*}
    &\int_0^T \tau_t\begin{pmatrix}
        \sinh(\tau_t)\\
        \cosh(\tau_t)
    \end{pmatrix}(\cosh(\tau_t)-1,\sinh(\tau_t)) dt\\
    &=\frac{1}{\sqrt{\tau}}\begin{pmatrix}
        \frac{\tau_T}{4}\cosh(2\tau_T)-\frac{\sinh(2\tau_T)}{8}-\tau_T\cosh(\tau_T)+\sinh(\tau_T) & \frac{\tau_T}{4}\sinh(2\tau_T)-\frac{\sinh^2(\tau_T)}{4}-\frac{\tau_T^2}{4}\\
        -\frac{7}{8}+\frac{\tau_T^2}{4}+\frac{\tau_T}{4}\sinh(2\tau_T)-\tau_T\sinh(\tau_T)-\frac{\cosh(2\tau_T)}{8}+\cosh(\tau_T) & \frac{\tau_T}{4}\cosh(2\tau_T)-\frac{\sinh(2\tau_T)}{8}
    \end{pmatrix}\\
    &=:\frac{B_1}{\sqrt{\tau}},\\
    &\int_0^T \begin{pmatrix}
        \sinh(\tau_t)\\
        \cosh(\tau_t)
    \end{pmatrix}(\sinh(\tau_t)-\tau_t,\cosh(\tau_t)-1) dt\\
    &=\frac{1}{\sqrt{\tau}}\begin{pmatrix}
        \frac{\sinh(2\tau_T)}{4}-\frac{\tau_T}{2}-\tau_T\cosh(\tau_T)+\sinh(\tau_T) & \frac{(\cosh(\tau_T)-1)^2}{2}\\
        \frac{\sinh^2(\tau_T)}{2}-\tau_T\sinh(\tau_T)+\cosh(\tau_T)-1 & \frac{\tau_T}{2}+\frac{\sinh(2\tau_T)}{4}-\sinh(\tau_T)
    \end{pmatrix}\\
    &=:\frac{B_2}{\sqrt{\tau}}.
\end{align*}
Then, the integration $\int_0^T\langle \ddot{x},x^\prime\rangle dt$ becomes
\begin{align*}
    \frac{\tau^\prime}{2\sqrt{\tau}}c^T\left(B_1-\frac{2\tau}{\tau^\prime}B_2A^{-1}A^\prime\right)c,
\end{align*}
where $\tau^\prime=\frac{1}{(1-w)^2}> 0$ for $w\neq 1$ and
\begin{align*}
    B_1-\frac{2\tau}{\tau^\prime}B_2A^{-1}A^\prime=:B_3=:\begin{pmatrix}
        b_{11} & b_{12}\\
        b_{21} & b_{22}
    \end{pmatrix},
\end{align*}
\begin{align*}
    b_{11}=&-\tau_T \cosh(\tau_T)+\frac{\tau_T}{4}\cosh(2\tau_T)+\sinh(\tau_T)-\frac{1}{8}\sinh(2\tau_T)+\\
    &\frac{\tau_T}{2\tilde{\Delta}}(\cosh(\tau_T)-1)\sinh(\tau_T)(\tau_T-2\sinh(\tau_T)-\cosh(\tau_T)(\sinh(\tau_T)-2\tau_T)),\\
    b_{12}=&-\frac{\tau_T^2}{4}-\frac{\sinh^2(\tau_T)}{4}+\frac{\tau_T\sinh(2\tau_T)}{4}+\frac{1}{2\tilde{\Delta}}(\cosh(\tau_T)-1)^3\sinh(\tau_T),\\
    b_{21}=&-\frac{7}{8}+\frac{\tau_T^2}{4}+\cosh(\tau_T)-\frac{\cosh(2\tau_T)}{8}-\tau_T\sinh(\tau_T)+\frac{\tau_T}{4}\sinh(2\tau_T)-\\
    &\frac{\tau_T}{2\tilde{\Delta}}(\cosh(\tau_T)-1+\tau_T\sinh(\tau_T))(2-2\cosh(\tau_T)+2\tau_T\sinh(\tau_T)-\sinh^2(\tau_T)),\\
    b_{22}=&\frac{\tau_T}{4}\cosh(2\tau_T)-\frac{\sinh(2\tau_T)}{8}-\\
    &\frac{1}{4\tilde{\Delta}}(\sinh(\tau_T)-\tau_T)(\sinh(\tau_T)+\tau_T\cosh(\tau_T))(2\tau_T-4\sinh(\tau_T)+\sinh(2\tau_T)),\\
    \tilde{\Delta}=&\Delta/\tau_T.
\end{align*}
It is a bit messy to write down the closed forms of the eigenvalues of $B_3$, instead, we plot the eigenvalues for straightforward visualization. In Figure \ref{fig:Euclidean_eigen}, we choose $T=1$, from which we can see that eigenvalues of $B_3$ are positive at least on $(0.2,1)$. Therefore, the convexity condition \ref{convex_ws} is verified at least on $(0.2,1)$.

    \begin{figure}[ht]
    \centering
    \includegraphics[width=75mm]{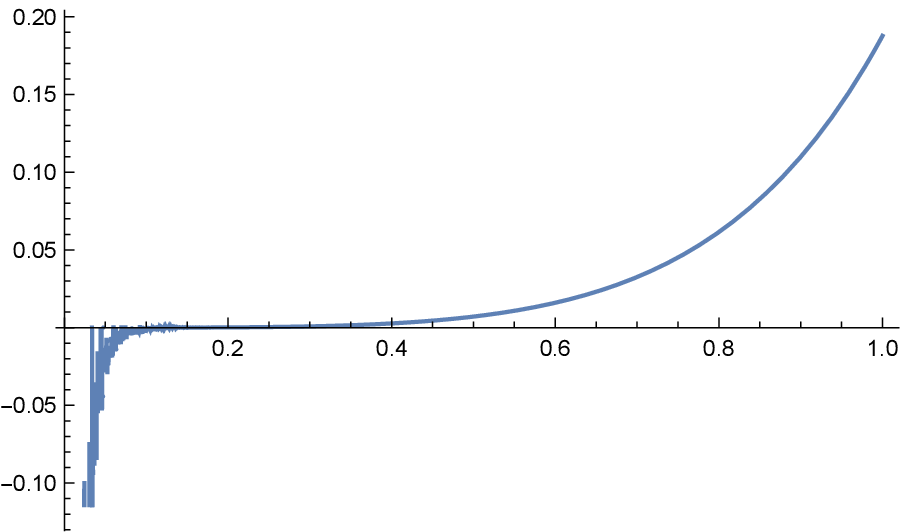}
    \includegraphics[width=75mm]{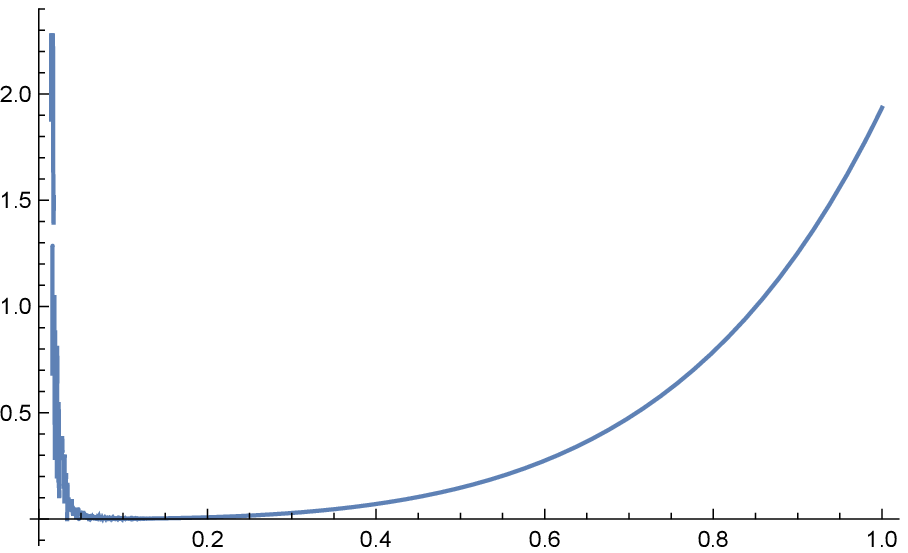} \\[5mm]
    \caption{\sf Eigenvalues of $B_3$ (Left: smaller, Right: larger)}\label{fig:Euclidean_eigen}
    \end{figure}

\section{Riemannian cubics in tension on torus}
\label{app:B}

In this section, we consider Riemannian cubics in tension on torus using variational methods. To short notations, we denote by $p:=\frac{2a\sin v}{c+a\cos v}$, $q:=\sin v\left(\frac{c}{a}+\cos v\right)$, $r:=c+a\cos v$ and their derivatives with respect to $v$ by upper notation prime. Then, we have
\begin{align*}
    F=&\int_0^T\Vert\nabla_t\dot{x}\Vert^2+\lambda \Vert\dot{x}\Vert^2 dt\\
    =&\int_0^T(\ddot{u}-p\dot{u}\dot{v})^2+(\ddot{v}+q\dot{u}^2)^2+\lambda(r^2\dot{u}^2+a^2\dot{v}^2) dt.
\end{align*}

Considering the variation of $F$, we find
\begin{align*}
    \delta F=&2\int_0^T(\ddot{u}-p\dot{u}\dot{v})(\delta\ddot{u}-p^\prime\delta v\dot{u}\dot{v}-p\delta\dot{u} \dot{v}-p\dot{u}\delta\dot{v})+(\ddot{v}+q\dot{u}^2)(\delta\ddot{v}+q^\prime\delta v\dot{u}^2+2q\dot{u}\delta\dot{u})+\\
    &\lambda(rr^\prime\delta v\dot{u}^2+r^2\dot{u}\delta\dot{u}+a^2\dot{v}\delta\dot{v})dt\\
    =&2\int_0^T\left(\frac{d^2}{dt^2}\left(\ddot{u}-p\dot{u}\dot{v}\right)+\frac{d}{dt}\left(\left(\ddot{u}-p\dot{u}\dot{v}\right)p\dot{v}\right)-\frac{d}{dt}\left(\left(\ddot{v}+q\dot{u}^2\right)2q\dot{u}\right)-\frac{d}{dt}\left(\lambda r^2\dot{u}\right)\right)\delta u+\\
    &\left(\frac{d^2}{dt^2}\left(\ddot{v}+q\dot{u}^2\right)+p\frac{d}{dt}\left((\ddot{u}-p\dot{u}\dot{v})\dot{u}\right)+(\ddot{v}+q\dot{u}^2)q^\prime \dot{u}^2+\lambda rr^\prime\dot{u}^2-\frac{d}{dt}\left(\lambda a^2\dot{v}\right)\right)\delta v dt,
\end{align*}
which implies the following Euler--Lagrange equations,
\begin{align*}
    u^{(4)}&-(p^2+p^\prime)\ddot{u}\dot{v}^2-(p+2q)\ddot{u}\ddot{v}-(p^{\prime\prime}+2pp^\prime)\dot{u}\dot{v}^3-(2p^2+3p^\prime+2q^\prime)\dot{u}\dot{v}\ddot{v}\\
    &-(p+2q)\dot{u}\dddot{v}-4qq^\prime \dot{u}^3\dot{v}-6q^2\dot{u}^2\ddot{u}-2\lambda rr^\prime\dot{u}\dot{v}-\lambda r^2\ddot{u}=0,\\
    v^{(4)}&+(q^{\prime\prime}-pp^\prime)\dot{u}^2\dot{v}^2+(2q^\prime-p^2)(2\dot{u}\ddot{u}\dot{v}-\dot{u}^2\ddot{v})+(p+2q)(\ddot{u}^2+\dot{u}\dddot{u})\\
    &+qq^\prime\dot{u}^4+\lambda rr^\prime\dot{u}^2-\lambda a^2\ddot{v}=0.
\end{align*}

Introducing the variables $u_1:=u$, $u_2:=\dot{u}$, $u_3:=\ddot{u}$, $u_4:=\dddot{u}$, $v_1:=v$, $v_2:=\dot{v}$, $v_3:=\ddot{v}$, $v_4:=\dddot{v}$, we have the following first-order ODE,
\begin{align}
\begin{cases}
\dot{u}_1=&u_2,\\
\dot{u}_2=&u_3,\\
\dot{u}_3=&u_4,\\
\dot{u}_4=&(p^2+p^\prime)u_3v_2^2+(p+2q)u_3v_3+(p^{\prime\prime}+2pp^\prime)u_2v_2^3+(2p^2+3p^\prime+2q^\prime)u_2v_2v_3\\
    &+(p+2q)u_2v_4+4qq^\prime u_2^3v_2+6q^2u_2^2u_3-2\lambda a^2qu_2v_2+\lambda r^2u_3,\\
\dot{v}_1=&v_2,\\
\dot{v}_2=&v_3,\\
\dot{v}_3=&v_4,\\
\dot{v}_4=&(pp^\prime-q^{\prime\prime})u_2^2v_2^2+(p^2-2q^\prime)(2u_2u_3v_2-u_2^2v_3)-(p+2q)(u_3^2+u_2u_4)\\
    &-qq^\prime u_2^4-\lambda rr^\prime u_2^2+\lambda a^2v_3.
\end{cases}
\end{align}
Therefore, in order to solve the problem (T$_{ws}$), we only have to solve the above ODE system.

\end{document}